\documentclass[11pt]{amsart}
\usepackage{amsmath}
\usepackage{amsthm}
\usepackage{amssymb}
\newtheorem{theorem}{Theorem}[section]
\newtheorem{prop}[theorem]{Proposition}
\newtheorem{lemma}[theorem]{Lemma}
\newtheorem{corollary}[theorem]{Corollary}
\newtheorem{conjecture}[theorem]{Conjecture}

\newtheorem{observation}[theorem]{Observation}

\newtheorem{question}[theorem]{Question}
\newtheorem{remark}[theorem]{Remark}

\newcommand{\ZZ}{{\mathbb {Z}}}

\newcommand{\cyc}{{\operatorname{cyc}}}

\newcommand{\odd}{{\operatorname{odd}}}
\newcommand{\even}{{\operatorname{even}}}

\usepackage[colorlinks=true, allcolors=blue]{hyperref}

\dedicatory{Dedicated to Gil Kalai, a long time friend, collaborator, mentor and a source of inspiration, \\
on the occasion of his 70th birthday}
\title{Circular Sorting}
\author{Ron M.\ Adin}
\address{Department of Mathematics, Bar-Ilan University, Ramat-Gan 52900, Israel}
\email{radin@math.biu.ac.il}

\author{Noga Alon}
\address{Department of Mathematics, 
Princeton University, Princeton, NJ 08544, USA}
\email{nalon@math.princeton.edu}
\thanks{The second author was supported in part by NSF grant DMS-2154082.%
}

\author{Yuval Roichman}
\address{Department of Mathematics, Bar-Ilan University, Ramat-Gan 52900, Israel}
\email{yuvalr@math.biu.ac.il}

\date{August 5, 2025}

\subjclass{05A05, 68P10}
\keywords{Sorting, adjacent transpositions, cyclic permutation, random permutation, multiplicative group modulo $n$}

\begin{document}

\begin{abstract}
    We determine the maximal number of steps required to sort $n$ labeled points on a circle by adjacent swaps. Lower bounds for sorting by all swaps, not necessarily adjacent, are given as well.
\end{abstract}

\maketitle

\tableofcontents

\section{Introduction}
\vspace{0.2cm}


In this paper we determine the maximal number of steps required to sort 
$n$ labeled points on a circle by adjacent swaps. 
This problem was explored, for example, in the context of micro-rearrangements of gene order in viruses; 
see, e.g.,~\cite{Chen}.
The analogous problem for labeled points on a line 
is classical. 
In fact, permutation sorting by adjacent transpositions may be traced back to early works in combinatorial group theory. 
Since Cayley graphs are vertex transitive, the (worst case) sorting time of a permutation by a set of involutions (generating the symmetric group) is equal to the diameter of the corresponding Cayley graph; equivalently, to the maximal length of an element in the prescribed generating set. 
Of special interest is sorting by adjacent transpositions, where the diameter is  
the maximal Coxeter length of an element. For any finite reflection group, 
the maximal Coxeter length is equal 
to the number of all reflections~\cite[\S 1.7]{Humphreys}. In the symmetric group $S_n$, this is the number of all transpositions, namely $\binom{n}{2}$.  
For sorting permutations by other sets of involutions see, 
e.g., \cite{GT}, \cite{Gates}, \cite{Je}, \cite{Feng}, 
\cite{ZBSY} and follow-ups. 
The last three papers among these five discuss the problem of sorting by all 
adjacent transpositions together with the transposition $(1,n)$, which
is closely related (though not identical) to one of 
the problems we consider here. 


A {\em cyclic permutation} of order $n$, also called circular permutation, is an equivalence class of arrangements of the numbers $1, \ldots, n$ on a circle, where all cyclic shifts of an arrangement are considered equivalent~\cite{Brualdi}. 
Permutations are linearly ordered sets, while   
cyclic permutations are ``cyclically ordered'' sets. 
An axiomatic approach to cyclic orders was developed by Huntington~\cite{Huntington}, and extended to partial cyclic orders by Megiddo~\cite{Megiddo}.
This concept was extensively studied and used since then.  
A different type of cyclic order, called toric order, was introduced by Develin, Macauley and Reiner~\cite{toric}.  
Total 
cyclic orders, as well as total toric orders, on the set $[n]$ correspond to the $(n - 1)!$ cyclic permutations. 





Let $\pi$ be a cyclic permutation of order $n$, 
represented by a labeling of the vertices of a cycle of length $n$ by the elements of $[n]=\{1,2,\ldots,n\}$ in a bijective way. 
An {\em adjacent swap} is any labeling obtained from $\pi$ by swapping the labels of two adjacent vertices along the cycle. 
How many adjacent swaps are needed in order to convert $\pi$ into the trivial cyclic permutation $u=(1,2,\ldots ,n)$?   
Equivalently, what is the diameter of the graph on all cyclic permutations represented as above,
in which two cyclic permutations are adjacent 
if and only if one can be obtained from the other by a single adjacent swap?  
This graph is vertex transitive, and hence its diameter is indeed the maximum, over all cyclic permutations $\pi$, of the minimum number of adjacent swaps required to transform $\pi$  into $u$. 
Let $f(n)$ denote this diameter. We prove the following.

\begin{theorem}\label{t:main}
    For any $n \ge 1$,
    \begin{equation}\label{eq:main}
            f(n) = \left\lfloor \frac{(n-1)^2}{4} \right\rfloor .
    \end{equation}
\end{theorem}

It is worth noting that in~\cite{ZBSY} it is proved that any permutation can be sorted by at most $\left\lfloor \frac{n^2}{4} \right\rfloor$ adjacent cyclic transpositions, and this is tight. Although this appears to be closely related to the statement of Theorem \ref{t:main}, we do not see a way to derive either of these two results from the other, and the proofs are completely different.

The rest of the paper is organized as follows. 
In Subsection~\ref{sec:adjacent_upper} we prove that the RHS of Equation~\eqref{eq:main} is an upper bound for the sorting time of a cyclic permutation by adjacent swaps.   
In Subsection~\ref{sec:adjacent_lower} we prove that this upper bound is tight, completing the proof of Theorem~\ref{t:main}. 
Finally, in Section~\ref{sec:all} we obtain bounds for the sorting time of cyclic permutations by all possible swaps, not necessarily adjacent.

\section{Sorting by adjacent swaps}\label{sec:adjacent}

In this section,
all swaps mentioned are adjacent swaps.

\subsection{Upper bound}\label{sec:adjacent_upper}

In this subsection we prove the following. 

\begin{theorem}\label{t:upper_bound}
   For any $n \ge 1$,
    \[
        f(n) \le \left\lfloor \frac{(n-1)^2}{4} \right\rfloor .
    \]
\end{theorem}


\begin{proof} 
    We distinguish four cases, according to the remainder of $n$ upon division by $4$.
    \begin{enumerate}
    
    \item[(a)]
    Assume first that $n$ is divisible by $4$: $n = 4m$. 
    Let $\pi$ be a cyclic permutation, and consider its representation as a labeled cycle $C$. Call a label small if it lies in $\{1, 2, \ldots, 2m\}$ and large if it lies in $\{2m+1, 2m+2, \ldots ,4m\}$. 
    We claim that there is a partition of $C$ into two arcs, $A_1$ and $A_2$, each being a path of $2m$ consecutive vertices along $C$, so that each $A_i$ contains exactly $m$ small labels and $m$ large labels. 
    Indeed, starting with any arc $A$ of $2m$ vertices, when we shift it by one step the number of small labels can change by at most  $1$. 
    If $A$ contains at most $m$ small labels then its complement contains at least $m$ small labels, and thus the assertion of the claim follows by the discrete intermediate value theorem.

    Fix two disjoint arcs, $A_1$ and $A_2$, each containing exactly $m$ small and $m$ large labels. Think about these two arcs as the left and right halves of the cycle $C$, and subdivide each of them into a top quarter and a bottom quarter.
    We now consider two possible sequences of swaps.
    The first shifts all the $m$ small labels of $A_1$ to its top quarter and all $m$ small labels of $A_2$ to its top quarter, 
    while the second shifts all the small labels of $A_1$ to its bottom quarter and all those of $A_2$ to its bottom quarter. 
    Let the vertices of $A_1$ be $v_1,v_2, \ldots ,v_{2m}$, in order, from top to bottom. 
    Let the small labels in $A_1$ be $s_1, s_2, \ldots ,s_{m}$, also ordered from top to bottom. 
    Then, in the first process, $s_1$ has to move to the vertex $v_1$, and in the second process it has to move to vertex $v_{m+1}$. Similarly, $s_2$ has to reach vertex $v_2$ in the first process and vertex $v_{m+2}$ in the second, and in general $s_i$ has to reach vertex $v_{i}$ in the first process and vertex $v_{i+m}$ in the second. 
    It is easy to check that the current location of each label $s_i$ is between these two destinations it has to reach (since there are $i-1$ small labels in $A_1$ before it and $m-i$ after it). 
    Therefore, the sum of the two distances from the location of $s_i$ to its two destinations in the two possible processes above is $m$.  
    The total number of swaps required to shift all small labels in $A_1$ to its top part is exactly the sum of distances of these labels from their destinations in the first process, and the analogous statement holds for the number of swaps required to shift them to the bottom part.
    The same reasoning applies, of course, to the arc $A_2$ as well. 
    Let $f_1$ denote the total number of swaps required for the first process (shifting all small labels in $A_1$ to its top part and all those in $A_2$ to its top part),
    and let $f_2$ denote the total number of swaps required for the second process. 
    By the discussion above it follows that
    \[
        f_1 + f_2
        = m \cdot m + m \cdot m 
        = 2m^2,
    \]
    since the total number of swaps in both processes in which any fixed small label moves is exactly $m$.

    Let $S_1$ denote the set of $m$ small labels in $A_1$, let $L_1$ denote the set of $m$ large labels in $A_1$, and define $S_2,L_2$ analogously for the arc $A_2$. 
    In order to complete the sorting after the first process, it suffices to sort the labels in $S_1 \cup S_2$ that lie in the top half of the circle, and sort the labels $L_1 \cup L_2$ in the bottom half. 
    The number of swaps required to do so is the number of inversions in $S_1 \cup S_2$ in the order they ended at the top, plus the number of inversions in $L_1 \cup L_2$ in the order they ended in the bottom. Denote this number by $g_1$.  
    Similarly, in order to complete the sorting after the second process it suffices to sort the labels in $S_1 \cup S_2$ that ended in the bottom  and the labels in $L_1 \cup L_2$ that ended in the top. 
    Let this sum be $g_2$. 
    The crucial observation now is that each pair of labels $s_1 \in S_1$ and $s_2 \in S_2$ form an inversion at the end of the first process  if and only if they do not form an inversion in the second. Therefore, the sum $g_1+g_2$ satisfies
    \[
        g_1 + g_2 
        \le 2 \cdot 4 \binom{m}{2} + m \cdot m + m \cdot m
        = 6m^2 - 4m.
    \]
    Indeed, the total number of inversions between pairs of labels in $S_1$, in $S_2$, in $L_1$ and in $L_2$, which are counted twice in $g_1+g_2$, is at most $2 \cdot 4 \binom{m}{2}$. 
    The total number of inversions between pairs of labels with one in $S_1$ and one in $S_2$ contributes to the sum $g_1+g_2$ only $m \cdot m$, since each such pair forms an inversion only once, and the same applies for pairs with one label in  $L_1$ and one in $L_2$.

    Summing the equality (for $f_1 + f_2$) and inequality (for $g_1 + g_2$) above, we get
    \[
        f_1 + f_2 + g_1 + g_2
        \le 8m^2 - 4m.
    \]
    Therefore either $f_1+g_1$ or $f_2+g_2$ is at most 
    \[
        4m^2 - 2m = \frac{n^2 - 2n}{4} 
        = \left\lfloor \frac{(n-1)^2}{4} \right\rfloor. 
    \]
    As we can complete the sorting with $f_1+g_1$ swaps, and also with $f_2+g_2$ swaps, this completes the proof for $n$ divisible by $4$. 

    \item[(b)]
    Assume now that $n = 4m+2$. 
    Following the previous proof, with the necessary adaptations, call a label small if it lies in $\{1,2,\ldots,2m+1\}$ and large if it lies in $\{2m+2,2m+3,\ldots,4m+2\}$.
    By the same argument as before, there is an arc $A_1$ of $2m+1$ consecutive points with exactly $m$ small (and $m+1$ large) labels.
    Its complement $A_2$ is also an arc of length $2m+1$, with $m+1$ small (and $m$ large) labels.
    View $A_1$ ($A_2$) as the left (respectively, right) half of the whole cycle.
    Consider two possible sequences of swaps: one shifting all the small labels of $A_1$ to its top, and all the small labels of $A_2$ to its top; and the other shifting all small labels of each arc to its bottom.
    The $i$-th small label in $A_1$ ($1 \le i \le m$, counting from top to bottom) will move to either the $i$-th or the $(m+1+i)$-th position in $A_1$ (again, counting from top to bottom). Its current position is somewhere between $i$ and $m+1+i$, and therefore the sum of its distances from the two possible endpoints is $m+1$.
    A similar argument for $A_2$, with $1 \le i \le m+1$ and positions $i$ and $i+m$, yields sum of distances $m$. If $f_1$ ($f_2$) is the total number of swaps required for the first (respectively, second) process, then
    \[
        f_1 + f_2 
        =  m \cdot (m+1) + (m+1) \cdot m  
        = 2m(m+1). 
    \]
    Let $S_1$ denote the set of $m$ small labels in $A_1$, let $L_1$ denote the set of $m+1$ large labels in $A_1$, and define analogously $S_2$ and $L_2$ (of sizes $m+1$ and $m$, respectively) for the arc $A_2$. 
    Denote by $g_1$ ($g_2$) the number of swaps require to sort $S_1 \cup S_2$ as well as $L_1 \cup L_2$ in the first (respectively, second) process.
    The crucial observation is that each pair of labels $s_1 \in S_1$ and $s_2 \in S_2$ form an inversion at the end of the first process if and only if they do not form an inversion in the second;
    and a similar claim for $\ell_1 \in L_1$ and $\ell_2 \in L_2$. 
    Therefore
    \[
        g_1 + g_2 
        \le 2 \cdot \left( 2 \binom{m}{2} + 2 \binom{m+1}{2} \right) + 2m(m+1)
        = 6m^2 + 2m.
    \]

    Summing the equality (for $f_1 + f_2$) and inequality (for $g_1 + g_2$) above, we get
    \[
        f_1 + f_2 + g_1 + g_2
        \le 8m^2 + 4m.
    \]
    Therefore either $f_1+g_1$ or $f_2+g_2$ is at most 
    \[
        4m^2 + 2m = \frac{(n-2)^2 + 2(n-2)}{4} 
        = \frac{n^2 - 2n}{4}
        = \left\lfloor \frac{(n-1)^2}{4} \right\rfloor. 
    \]
    This completes the proof for $n = 4m+2$. 


    \item[(c)]
    Next assume that $n = 4m+1$.  
    Call a label small if it lies in $\{1,2,\ldots,2m\}$ and large if it lies in $\{2m+1,2m+2,\ldots,4m+1\}$.
    Then there is an arc $A_1$, consisting of $2m$ consecutive points, with exactly $m$ small (and $m$ large) labels.
    Its complement $A_2$ is an arc of length $2m+1$, with $m$ small (and $m+1$ large) labels. 
    Then, using notations as before,
    \[
        f_1 + f_2 
        =  m \cdot m + m \cdot (m+1)
        = 2m^2 + m
    \]
    and
    \[
        g_1 + g_2 
        \le 2 \cdot \left( 3 \binom{m}{2} +  \binom{m+1}{2} \right) + m^2 + m(m+1)
        = 6m^2 - m.
    \]
    Therefore either $f_1+g_1$ or $f_2+g_2$ is at most 
    \[
        4m^2 = \frac{(n-1)^2}{4}. 
    \]
    This completes the proof for $n = 4m+1$.

    \item[(d)]
    Finally, assume that $n = 4m+3$.  
    Call a label small if it lies in $\{1,2,\ldots,2m+2\}$ and large if it lies in $\{2m+3,2m+4,\ldots,4m+3\}$.
    By a slightly modified argument (considering only arcs that do not contain a specific point with a large label), there is an arc $A_1$, consisting of $2m+1$ consecutive points, with exactly $m+1$ small (and $m$ large) labels.
    Its complement $A_2$ is an arc of length $2m+2$, with $m+1$ small (and $m+1$ large) labels. 
    Then, using notations as before,
    \[
        f_1 + f_2 
        = (m+1) \cdot m + (m+1) \cdot (m+1)
        = 2m^2 + 3m + 1
    \]
    and
    \[
        g_1 + g_2 
        \le 2 \cdot \left( 3\binom{m+1}{2} + \binom{m}{2} \right) +  (m+1)^2 + m(m+1)
        = 6m^2 + 5m + 1.
    \]
    Therefore either $f_1+g_1$ or $f_2+g_2$ is at most 
    \[
        4m^2 + 4m +1 = (2m+1)^2
        = \frac{(n-1)^2}{4}. 
    \]
    This completes the proof for $n = 4m+1$.


\end{enumerate}
Having dealt with all four cases, we have now completed the proof of Theorem~\ref{t:upper_bound}.
\end{proof}

\subsection{Lower bound}\label{sec:adjacent_lower}

Recall that any swap mentioned in this section is an adjacent swap.

In this subsection we prove

\begin{theorem}\label{t:lower_bound} 
    For any $n \ge 1$,
    \[
        f(n) \ge \left\lfloor \frac{(n-1)^2}{4} \right\rfloor .
    \]
\end{theorem}

\begin{proof}
By induction on $n$, carried out separately for odd and even values of $n$.

Assume first that $n$ is odd.
We shall prove that any circular sorting of the cyclic permutation 
$(n,n-1,\ldots,1)$ to the trivial cyclic permutation $u = (1,2,\ldots,n)$ 
requires at least $(n-1)^2/4$ swaps. 
This implies the claimed lower bound on the diameter $f(n)$.

The claim clearly holds for $n = 1$.
For the induction step, assume that $n > 1$ is odd, and that the claim holds for $n-2$.
Consider an arbitrary circular sorting of $(n,n-1,\ldots,1)$ to $u$. 
This can be viewed as a (non-circular) sorting to the identity permutation, using affine swaps from the set $\{(1,2), (2,3), \ldots,(n-1,n),(n,1)\}$, of a suitable cyclic shift of the permutation $w$, 
which is defined by $w(i) := n+1-i$ $(1 \le i \le n)$. 
Such a cyclic shift is a permutation $w_{n,k}: [n] \to [n]$ of the form 
\[
    w_{n,k}(i) \equiv k - i \pmod n
    \qquad (\forall i),
\]
for some fixed integer $k$. 
Observe that, for odd $n$ and any $k$, $w_{n,k}$ is an involution with one fixed point.  

Let $d$ denote circular distance on the set of $n$ points on the circle, namely
\[
    d(i,j) := \min\{|i-j|, n - |i-j|\}
    \qquad (\forall i,j \in [n]).
\]
For any $i \in [n]$, define the {\em gap} of $i$ to be the distance $d(i,w_{n,k}(i))$. 
The gap of the unique fixed point is $0$. 
The other gap values are the integers $1 \le d \le (n-1)/2$, each attained by exactly two points forming a 2-cycle of $w_{n,k}$.
Let $i$ and $j = w_{n,k}(i)$ be the two points having maximal gap $d = (n-1)/2$.
Denote $A := \{i,j\}$ and $B := [n] \setminus A$; clearly, $A$ and $B$ are invariant under $w_{n,k}$. The restriction of $w_{n,k}$ to $B$ can be viewed, 
after a suitable relabeling of the points, 
as the permutation $w_{n-2,\ell}$ for some $\ell$.

Consider now a sequence of swaps that sorts $w_{n,k}$ to the identity permutation. 
Distinguish two kinds of swaps: those involving only elements of $B$, and those that involve at least one element of $A$.
The latter swaps do not change the relative order of the elements of $B$, and therefore
the former swaps yield a sorting of $w_{n-2,\ell}$ to the identity permutation.
By the induction hypothesis, this requires at least $(n-3)^2/4$ swaps.
Each of the remaining swaps involves at least one element of $A$, and it is clear that this part of the sorting, eventually interchanging $i$ and $j$ which are at circular distance $(n-1)/2$, requires at least $2 \cdot (n-1)/2 -1 = n-2$ swaps. 
Altogether, the sequence contains at least 
\[
    \frac{(n-3)^2}{4} + (n-2) 
    = \frac{(n-1)^2}{4} 
\]
swaps, as claimed.
This proves the claim for any odd $n$.

Assume now that $n$ is even. 
Again, we can view a circular sorting of the cyclic permutation $(n,n-1,\ldots,1)$ to the cyclic permutation $u$ as a sorting to the identity permutation, using affine swaps, of a permutation $w_{n,k}: [n] \to [n]$ of the form
\[
    w_{n,k}(i) \equiv k - i \pmod n 
    \qquad (\forall i),
\]
for some fixed $k$.
This is again an involution, but now (for $n$ even) there are two options for its number of fixed points: $w_{n,k}$ has two fixed points if $k$ is even, and none if $k$ is odd.
We shall prove that the number of swaps needed to sort $w_{n,k}$ to the identity permutation is at least $N_{n,k}$, where
    \[
        N_{n,k} = 
        \begin{cases}
            (n^2-2n)/4, & \text{if } k - n/2 \text{ is odd;} \\
            (n^2-2n+4)/4, & \text{if } k - n/2 \text{ is even.}
        \end{cases}
    \]
Note that these numbers are $\lfloor (n-1)^2/4 \rfloor$ and $\lfloor (n-1)^2/4 \rfloor + 1$, respectively, so that our claim implies the required lower bound on the diameter $f(n)$.

The proof will proceed by induction on (even values of) $n$. 
The claim clearly holds for $n = 2$: $w_{2,k}$ is the identity permutation for even $k$ (with $N_{2,k} = 0$), and the non-identity element of $S_2$ for odd $k$ (with $N_{2,k} = 1$).

For the induction step, assume that $n > 2$ is even, and that the claim holds for $n-2$. We consider four cases, depending on the parity of $k$ and of $n/2$.
\begin{enumerate}
    \item[(a)]
    Assume that both $k$ and $n/2$ are even.
    The involution $w_{n,k}$ has two fixed points, and all its gaps are even.
    The minimal gap $0$ and maximal gap $n/2$ are each attained by two points, while each intermediate value $2 \le 2d \le (n-4)/2$ is attained by four points.
    Let $i$ and $j = w_{n,k}(i)$ be the two points with maximal (even) gap $n/2$; 
    denote $A := \{i,j\}$ and $B := [n] \setminus A$.
    Clearly, $A$ and $B$ are invariant under $w_{n,k}$. The restriction of $w_{n,k}$ to $B$ can be viewed, after a suitable relabeling of the points, as the permutation $w_{n-2,\ell}$ for some even $\ell$.
    Noting that $(n-2)/2$ is odd, we thus get at least
    \[
        N_{n-2,\ell} + (2 \cdot n/2 -1) 
        = \frac{(n-2)^2-2(n-2)}{4} + (n-1) 
        = \frac{n^2-2n+4}{4} = N_{n,k}
    \]
    swaps.

    \item[(b)]
    Assume that $k$ is even but $n/2$ is odd.
    The involution $w_{n,k}$ has two fixed points, and all its gaps are even.
    The minimal gap $0$ is attained by two points, while each other gap $2 \le 2d \le (n-2)/2$ is attained by four points.
    Let $i$ and $j = w_{n,k}(i)$ be two of the points with maximal (even) gap $(n-2)/2$; 
    denote $A := \{i,j\}$ and $B := [n] \setminus A$.
    Clearly, $A$ and $B$ are invariant under $w_{n,k}$. The restriction of $w_{n,k}$ to $B$ can be viewed, after a suitable relabeling of the points, as the permutation $w_{n-2,\ell}$ for some even $\ell$.
    Noting that $(n-2)/2$ is even, we thus get at least
    \[
        N_{n-2,\ell} + (2 \cdot (n-2)/2 -1) 
        = \frac{(n-2)^2-2(n-2)+4}{4} + (n-3) 
        = \frac{n^2-2n}{4} = N_{n,k}
    \]
    swaps.

    \item[(c)]
    Assume that $k$ is odd and $n/2$ is even.
    The involution $w_{n,k}$ has no fixed points, and all its gaps are odd.
    Each odd number $1 \le 2d-1 \le (n-2)/2$ is attained as a gap by four points.
    Let $i$ and $j = w_{n,k}(i)$ be two of the points with maximal (odd) gap $(n-2)/2$; 
    denote $A := \{i,j\}$ and $B := [n] \setminus A$.
    Clearly, $A$ and $B$ are invariant under $w_{n,k}$. The restriction of $w_{n,k}$ to $B$ can be viewed, after a suitable relabeling of the points, as the permutation $w_{n-2,\ell}$ for some odd $\ell$.
    Noting that $(n-2)/2$ is odd, we thus get at least
    \[
        N_{n-2,\ell} + (2 \cdot (n-2)/2 -1) 
        = \frac{(n-2)^2-2(n-2)+4}{4} + (n-3) 
        = \frac{n^2-2n}{4} = N_{n,k}
    \]
    swaps.

    \item[(d)]
    Assume that both $k$ and $n/2$ are odd.
    The involution $w_{n,k}$ has no fixed points, and all its gaps are odd.
    The maximal gap $n/2$ is attained by two points, while each other odd value $1 \le 2d-1 \le (n-4)/2$ is attained by four points. 
    Let $i$ and $j = w_{n,k}(i)$ be the two points with maximal (odd) gap $n/2$; 
    denote $A := \{i,j\}$ and $B := [n] \setminus A$.
    Clearly, $A$ and $B$ are invariant under $w_{n,k}$. The restriction of $w_{n,k}$ to $B$ can be viewed, after a suitable relabeling of the points, as the permutation $w_{n-2,\ell}$ for some odd $\ell$.
    Noting that $(n-2)/2$ is even, we thus get at least
    \[
        N_{n-2,\ell} + (2 \cdot n/2 -1) 
        = \frac{(n-2)^2-2(n-2)}{4} + (n-1) 
        = \frac{n^2-2n+4}{4} = N_{n,k}
    \]
    swaps.

\end{enumerate}

This proves our claim for any even $n$, and completes the proof of Theorem~\ref{t:lower_bound}.
\end{proof}

\begin{remark}
    {\rm Theorem \ref{t:main} determines the maximum number $f(n)$ of swaps needed to sort a cyclic permutation of $[n]$. 
    It may be interesting to determine or estimate the number of such cyclic permutations that can be sorted by $k$ swaps, for all $k \leq f(n)$.}
\end{remark}

\section{Sorting by all swaps}\label{sec:all}

The circular sorting question can be asked when a swap of any two (not necessarily adjacent)
elements is allowed. 
In this case 
$n-2$ swaps always suffice (Observation~\ref{obs:t0}).  
For $n=p^k$, where $p$ is an odd prime, at least $n-1-\log_p n$ swaps are required  (Proposition~\ref{t:prime_power}).  
It follows that for every prime $n$,  
the bound $n-2$ is tight (Corollary~\ref{cor:t_prime}). 
For general $n$ we get a lower bound of $n-O(\log n)$ 
(Theorem~\ref{t:all_lower}).  

\subsection{Upper bound}

Let $c := (1,2,\dots,n)$ be an $n$-cycle in $S_n$, and let $C_n = \langle c \rangle$ be the cyclic subgroup of $S_n$ generated by $c$. 
A cyclic permutation, namely an equivalence class $[\pi]$ where $\pi \in S_n$, may be identified with a coset $\pi C_n$.

Denote by $t([\pi])$ the sorting time of a cyclic permutation $[\pi]$ to the trivial cyclic permutation $[c]$, where a permissible step is a multiplication (on the left, or equivalently on the right) by any transposition, and let 
\[
    t(n) := \max_{\pi \in S_n} t([\pi]).
\]
Consider the graph on all cyclic permutations in which two cyclic permutations are adjacent if and only if one can be obtained from the other by a single a swap of any two letters. 
Since this graph is vertex transitive, its diameter is equal to  $t(n)$. 


For a permutation $\pi\in S_n$, let $\cyc(\pi)$ be the number of cycles in $\pi$. Observe that, for any $\pi\in S_n$, the sorting time of $[\pi]$ is
\[
    t([\pi]) 
    = \min_{\sigma\in \pi C_n} (n -\cyc(\sigma)).
 \]
Hence
\[
    t(n)
    = \max_{\pi\in S_n} \min_{\sigma\in \pi C_n} (n - \cyc(\sigma)).
\]

\begin{observation}\label{obs:t0}
    For every $n \ge 2$,
    \[
        t(n) \le n-2.
    \]
\end{observation}

\begin{proof}
    Every permutation has a cyclic shift with a fixed point, thus with at least 2 cycles.
\end{proof}

Computer experimentation show that, for $n \le 11$, this upper bound is attained only for prime values of $n$.  
In the following section we prove a general lower bound for prime powers, implying that the upper bound is indeed tight for every prime $n$. 



%




\subsection{Primes and prime powers} 

\begin{prop}\label{t:prime_power}
    If $n = p^k$, where $p$ is an odd prime and $k \ge 1$, then   
    \[
        t(n) \ge n-k-1 = n - \log_p n - 1.
    \]
\end{prop}

\begin{proof}
For $a\in \ZZ_n$ 
define a map $\pi_{n,a}:\ZZ_n \to \ZZ_n$ by
\[
    \pi_{n,a}(i) \equiv a \cdot i \pmod n
    \qquad (\forall\, i \in \ZZ_n). 
\]
Let $\ZZ_n^\times = \{a \in \ZZ_n \,:\, \gcd(a,n) = 1\}$ be the group of units of the ring $\ZZ_n$.
If $a \in \ZZ_n^\times$  then $\pi_{n,a}$ is a bijection; identifying $[n]$ with $\ZZ_n$, we can write $\pi_{n,a} \in S_n$. 
Moreover, if $\sigma \in \pi_{n,a} C_n$ then, for a suitable $j \in \ZZ_n$,
\[
    \sigma(i) \equiv a \cdot (i + j) \pmod n
    \qquad (\forall\, i \in \ZZ_n). 
\]
We want each $\sigma$ to have a fixed point. Thus we want to have, for every $j \in \ZZ_n$, a solution $i$ to the congruence
\[
    i \equiv a \cdot (i + j) \pmod n,
\]
namely to
\[
    (1-a) \cdot i \equiv a \cdot j \pmod n.
\]
This has a solution for every $j \in \ZZ_n$ (in particular, for $j = a^{-1}$) if and only if $a-1 \in \ZZ_n^\times$. 
Assuming, indeed, that $a,a-1 \in \ZZ_n^\times$, let 
\[
    i_\sigma \equiv (1-a)^{-1} \cdot a \cdot j \pmod n
\]
be the (unique) fixed point of $\sigma$.
Then
\[
    \sigma(i) - i_\sigma 
    = \sigma(i) - \sigma(i_\sigma) 
    = a \cdot (i + j) - a \cdot (i_\sigma + j)
    = a \cdot (i  - i_\sigma)
    \qquad (\forall\, i \in \ZZ_n).
\]
Thus $\sigma$ is conjugate (in $S_n$) to $\pi_{n,a}$, by the permutation corresponding to a cyclic shift by $i_\sigma$, and therefore $\sigma$ and $\pi_{n,a}$ have the same number of cycles:
\[
    \cyc(\sigma) = \cyc(\pi_{n,a})
    \qquad (\forall \sigma \in \pi_{n,a}\ZZ_n).
\]
Now let $p$ be an odd prime and $k$ a positive integer.
The following facts regarding generators of the cyclic group $\ZZ_{p^k}^\times$ are part of the remarks following~\cite[Lemma 1.4.5]{Cohen}.
Let $1 < g_0 < p$ be a generator of the cyclic group $\ZZ_p^\times$, and define
\[
    g := 
    \begin{cases}
        g_0, &\text{if } g_0^{p-1} \not\equiv 1 \pmod {p^2}; \\
        g_0 + p, &\text{otherwise.}
    \end{cases}
\]
Then $g$ is a generator of $\ZZ_{p^k}^\times$, simultaneously for all $k \ge 1$.
The additive group $\ZZ_{p^k}$ is a disjoint union of its subsets $C_0, \ldots, C_k$, where $C_i = p^i \ZZ_{p^{k-i}}^\times$ $(0 \le i \le k-1)$ and $C_k = \{0\}$. 
For the above choice of a generator $g$, each of these sets forms a single cycle of $\pi_{p^k,g}$, thus $\cyc(\pi_{p^k,g}) = k+1$.
Note also that $\gcd(g-1,p) = 1$, thus $g-1 \in \ZZ_{p^k}^\times$ for any $k \ge 1$.
It follows that
\[
    \cyc(\sigma) = \cyc(\pi_{p^k,g}) = k+1
    \qquad (\forall \sigma \in \pi_{p^k,g}\ZZ_{p^k})
\]
and therefore
\[
    t(p^k)
    = \max_{\pi \in S_{p^k}} \min_{\sigma \in \pi\ZZ_{p^k}} (p^k - \cyc(\sigma))
    \ge \min_{\sigma \in \pi_{p^k,g}\ZZ_{p^k}} (p^k - \cyc(\sigma)) 
    = p^k - (k+1). \qedhere
\]
\end{proof}

\begin{corollary}\label{cor:t_prime}
        If $n$ is prime then
    \[
        t(n) = n-2.
    \]
\end{corollary}

\begin{proof}
This clearly holds for $n = 2$.
By Observation~\ref{obs:t0} and Proposition~\ref{t:prime_power}, it also holds if $n$ is an odd prime.
\end{proof}

\begin{conjecture}\label{conj:prime}
    $t(n) = n-2$ if and only if $n$ is prime.
\end{conjecture}

\begin{lemma}
    Assume that $n > 2$ and that $\pi \in S_n$ satisfies 
    \[
        \max_{\sigma \in \pi C_n} \cyc(\sigma) = 2.
    \]
    Then each $\sigma \in \pi C_n$ has cycle structure $(n-1,1)$.
\end{lemma}

\begin{proof}
    Clearly, each $j \in [n]$ is a fixed point of a unique cyclic shift $\sigma \in \pi C_n$. This defines a mapping $f: [n] \to \pi C_n$. If $\cyc(\sigma) = 2$ then $\sigma$ cannot have more than one fixed point (since $n > 2$). It follows that $f$ is injective, thus bijective, and each $\sigma \in \pi C_n$ has a unique fixed point. Its other cycle is, of course, of length $n-1$.
\end{proof}

\begin{conjecture}\label{conj:one_fixed_point}
    If $n > 2$ and $\pi \in S_n$ satisfies 
    \[
        \max_{\sigma \in \pi C_n} \cyc(\sigma) = 2,
    \]
    then $n$ is prime and 
    \[
        \pi(i) \equiv a \cdot i \pmod n
        \qquad (\forall i)
    \]
    for some $a \in \ZZ_n^\times$.
\end{conjecture}

Conjecture~\ref{conj:one_fixed_point} implies Conjecture~\ref{conj:prime}. Both have been verified for $n \le 11$.


By  similar arguments one can prove the following statements.     

\begin{prop}
    If $n = 2p^k$, where $p$ is an odd prime and $k\ge 1$, then 
    \[
        t(n) \ge n- 2k - 2. 
    \]
\end{prop}

\begin{proof}
    Let $g$ be a simultaneous generator for $\ZZ_{p^t}^\times$ for all $t \ge 1$, as in the proof of Proposition~\ref{t:prime_power},
    and fix an integer $k \ge 1$. 
    Then either $g$ or $g + p^k$ (whichever is odd) is a generator of $\ZZ_{2p^k}^\times$~\cite[p. 26]{Cohen}. 
    It follows that there exists a simultaneous odd generator $g$ 
    for all $\ZZ_{2p^j}^\times$, $1 \le j \le k$.  
    Using the notation $\pi_{n,a}$ from the proof of Proposition~\ref{t:prime_power}, it is easy to see that $\cyc(\pi_{2p^k,g}) = 2k+2$, with cycles corresponding to the subsets
    $C_i^{\odd}= p^i \ZZ_{2p^{k-i}}^\times$ and $C_i^{\even}= 2 p^i \ZZ_{2p^{k-i}}^\times$, $0\le i\le k$.
    
    Consider now the shifts of $\pi_{2p^k,g}$, of the form $\sigma_j := \pi_{2p^k,g} c^j$ where $c = (1,2,\ldots,2p^k)$.
    We would like to have a fixed point for $\sigma_j$ for each value of $j$, since then $\sigma_j$ is conjugate to $\sigma_0 = \pi_{2p^k,g}$;  
    but this is impossible, since one of the two consecutive integers $g-1$ and $g$ must be even.
    We therefore only claim that $\sigma_j$ has a fixed point for {\em even} values of $j$.

    Indeed, a fixed point $i$ for $\sigma_{j}$ ($j$ even) must satisfy
    \[
        (1-g) \cdot i \equiv g \cdot j 
        \pmod {2p^k},
    \]
    or equivalently, since both $j$ and $g-1$ are even,
    \[
        \frac{1-g}{2} \cdot i \equiv g \cdot \frac{j}{2} 
        \pmod {p^k},
    \]
    This has a solution for each even $j$ if and only if $(g-1)/2$ is invertible in $\ZZ_{p^k}$, which is indeed the case (since otherwise $g \equiv 1 \pmod p$, contradicting the fact that $g$ generates $\ZZ_p^\times$). 
    Thus indeed $\sigma_j$ has a fixed point, and is thus conjugate to $\sigma_0$, implying that
    \[
        \cyc(\pi_{2p^k,g} c^j) = \cyc(\pi_{2p^k,g}) = 2k+2 
        \qquad (\forall \text{ even } j).
    \]
    
    Now consider odd values of $j$. Since
    \[
        \sigma_j(i) \equiv g \cdot (i + j) \pmod {2p^k},
    \]
    and $g$ is odd, the number $\sigma_j(i)$ is odd if and only if $i$ is even. It follows that each cycle of $\sigma_j$ alternates between even and odd numbers, and therefore has even length. It follows that each cycle of $\sigma_j$ splits into two cycles of $\sigma_j^2$, thus 
    \[
        \cyc(\sigma_j^2) = 2 \cdot \cyc(\sigma_j)
        \qquad (\forall \text{ odd } j).
    \]
    On the other hand, we claim that, for any value of $j$, $\sigma_j^2$ has a fixed point and is thus conjugate to $\sigma_0^2$. 
    Indeed,
    \[
        \sigma_0^2(i) \equiv g^2 \cdot i
        \pmod {2p^k}
        \qquad (\forall i)
    \]
    while
    \[
        \sigma_j^2(i) \equiv g^2 \cdot i + (g^2 + g) \cdot j
        \pmod {2p^k}
        \qquad (\forall i).
    \]
    In order to have $\sigma_j^2(i) = i$ we need
    \[
        (1 - g^2) \cdot i \equiv (g^2 + g) \cdot j
        \pmod {2p^k}    
    \]
    or, equivalently,
    \[
        \frac{1+g}{2} \cdot (1 - g) \cdot i \equiv \frac{1+g}{2} \cdot g \cdot j
        \pmod {p^k}.    
    \]
    To this end, it is sufficient (though not necessary!) to have
    \[
        (1 - g) \cdot i \equiv g \cdot j
        \pmod {p^k}.  
    \]
    As noted above, this equation has a solution $i$ for each value of $j$, since $g-1 \in \ZZ_{p^k}^\times$. The (now established) existence of a fixed point $i_0$ for $\sigma_j^2$ implies that
    \[
        \sigma_j^2(i) - i_0 
        \equiv \sigma_j^2(i) - \sigma_j^2(i_0) 
        \equiv g^2 \cdot (i - i_0)
        \equiv \sigma_0^2(i - i_0)
        \pmod {2p^k},
    \]
    so that $\sigma_j^2 = c^{i_0} \sigma_0^2 c^{-i_0}$ is conjugate to $\sigma_0^2$ in $S_{2p^k}$. We therefore have
    \[
        \cyc(\sigma_j^2) = \cyc(\sigma_0^2)
        \qquad (\forall j).
    \]
    Finally, the above description of the cycles of $\sigma_0 = \pi_{2p^k,g}$ shows that all its cycle lengths are divisible by $p-1$, hence even, except for the two fixed points. It follows that
    \[
        \cyc(\sigma_0^2) = 2 \cdot \cyc(\sigma_0) - 2.
    \]
    Putting it all together, we have
        \[
        \cyc(\sigma_j) = \cyc(\sigma_0) - 1
        \qquad (\forall \text{ odd } j),
    \]
    namely
    \[
        \cyc(\pi_{2p^k,g} c^j) = \cyc(\pi_{2p^k,g}) - 1 = 2k+1 
        \qquad (\forall \text{ odd } j).
    \]
    Thus
    \begin{align*}
        t(2p^k)
        &= \max_{\pi \in S_{2p^k}} \min_{\sigma \in \pi\ZZ_{2p^k}} (2p^k - \cyc(\sigma))
        \ge \min_{\sigma \in \pi_{2p^k,g}\ZZ_{2p^k}} (2p^k - \cyc(\sigma)) \\
        &= \min(2p^k - (2k+2), 2p^k - (2k+1))
        = 2p^k- (2k+2). \qedhere
    \end{align*}
\end{proof}

\begin{prop}
\label{p38}
    If $n = 2^k$ and $k\ge 2$ then  
    \[
        t(n) \ge n- 2k + 1. 
    \]
\end{prop}

\begin{proof} 
    First, for $k=2$ the lower bound holds, since $t(4) = 1 = 4 - 2 \cdot 2 + 1$. 

    Assume that $k \ge 3$.  
    It was shown by Gauss~\cite[Art.\ 90--91]{Gauss} that, in this case, $\ZZ_{2^k}^\times \cong \ZZ_2 \times \ZZ_{2^{k-2}}$, 
    where $3$ is a generator of the factor $\ZZ_{2^{k-2}}$ and $-1$ is a generator of the factor $\ZZ_2$; 
    see also~\cite[Ch.~6,\,\S 6]{Vinogradov}. 
    
    
    Note that $3$ also generates the cyclic group $\ZZ_4^\times$.
    Using the notation $\pi_{n,a}$ from the proof of Proposition~\ref{t:prime_power}, 
    one concludes that the cycles of $\pi_{2^k,3}$ are the subsets 
    $C_i^\varepsilon := \{\varepsilon \cdot 2^i 3^j \,:\, 0 \le j < 2^{k-2-i}\}$ $(0\le i \le k-3,\, \varepsilon \in \{1,-1\})$, 
    $C_{k-2} = \{2^{k-2}, -2^{k-2}\}$, 
    $C_{k-1} = \{2^{k-1}\}$
    and $C_{k} = \{0\}$.
    Hence $\cyc(\pi_{2^k,3}) = 2k-1$. 

    Regarding the shifts, by arguments as above, for every even $j$, 
    $\pi_{2^k,3}c^j$ and $\pi_{2^k,3}$ are conjugate, hence have the same number of cycles. 
    Finally, we claim that for every odd $j$, $\pi_{2^k,3}c^j$ has two cycles.     
    Indeed, for $d \ge 1$ and any $m$ we have 
    \[
        (\pi_{2^k,3}c^j)^d(m)
        \equiv m \cdot 3^d + 3j \cdot \sum_{i=0}^{d-1} 3^i
        \equiv m \pmod {2^k} 
    \]
    if and only if
    \[
        (2m + 3j) \cdot \frac{3^{d}-1}{2}
        \equiv 0 \pmod {2^k},
    \]
    which holds (for an odd $j$) if and only if 
    \[
        3^{d}-1 \equiv 0 \pmod {2^{k+1}}. 
    \]
    Recalling that the order of $3$ in $\ZZ_{2^{k+1}}^\times$ is $2^{k-1}$, this holds if and only if $2^{k-1}$ divides $d$, so that indeed $\pi_{2^k,3}c^j$ for odd $j$ has two cycles of length $2^{k-1}$ each. Overall, for $k \ge 3$,
    \begin{align*}
        t(2^k)
        &= \max_{\pi \in S_{2^k}} \min_{\sigma \in \pi\ZZ_{2^k}} (2^k - \cyc(\sigma))
        \ge \min_{\sigma \in \pi_{2^k,3}\ZZ_{2^k}} (2^k - \cyc(\sigma)) \\
        &= \min(2^k - (2k-1), 2^k - 2)
        = 2^k- (2k-1). \qedhere
    \end{align*}
\end{proof}

\begin{question}
    Is the function $t(n)$ monotone?
\end{question}

If the function $t(n)$ is monotone, one can apply the above bounds to other integers. 
Note, however, that adding a fixed point to $\pi\in S_n$ may decrease the value of $t$.  
For example, $t([\pi_{11,3}]) = 9$, while defining $\sigma \in S_{12}$ by $\sigma(i) = \pi_{11,3}(i)$ for $i\le 11$ and $\sigma(12)=12$ yields $t([\sigma])=8$. 


\subsection{General lower bound}

In this section we prove a general lower bound, which is independent of the prime decomposition of $n$.

\begin{theorem}\label{t:all_lower}
   For any $n>1$ 
   \[
        t(n) \ge n - \lceil e \cdot (\ln n + 1) \rceil.   
   \]
\end{theorem}

Note that for large $n=2^k$ this bound is stronger than the one 
in Proposition \ref{p38}.

First we show that the probability that a random permutation has a large number of cycles is small.  
The following result is surely known in  a very precise form; 
for completeness we include a self contained elementary proof.

\begin{prop}\label{p31}
    For any $k \ge 0$, the probability that a random permutation in $S_n$ has exactly $k+1$ cycles is at most 
    \[
        \frac{(\ln (n-1) + 1)^k}{n \cdot k!} .
    \]
\end{prop}


\begin{proof}
    We count the number of permutations with $k+1$ cycles as follows.
    Write each cycle of the permutation with the smallest element of the cycle first, and arrange the cycles in increasing order of their first elements.
    The first cycle starts with the element $1$. We can select the next element of the cycle arbitrarily, then the next one, and so on. 
    If the length of this cycle is $n_1$ then this gives $(n-1)(n-2) \cdots (n-n_1+1)$ possibilities. 
    The second cycle must start with the smallest number that was not chosen yet. Proceed in the same manner to choose the other cycle elements. If its length is $n_2$ then this gives $(n-n_1-1)(n-n_1-2) \cdots (n-n_1-n_2+1)$ possibilities.
    Proceeding in this way, we see that the number of permutations with $k+1$ cycles of lengths $n_1, \ldots, n_{k+1}$ (arranged according to the above convention) is
    \[
        \frac{n!}{n(n-n_1)(n-n_1-n_2) \cdots 
    (n-n_1-n_2 \ldots -n_{k})}.
    \]
    Summing over all possibilities of cycle lengths gives 
    \[
        \frac{n!}{n} \cdot 
        \sum \frac{1}{m_1 m_2 \cdots m_{k}} \, ,
    \]
    where the summation is over all $k$-tuples $(m_1, m_2, \ldots, m_{k})$ of integers satisfying $n > m_1 > m_2 > \cdots > m_{k} > 0$.

    Consider now the expression 
    \[
        \left( \frac{1}{1} + \frac{1}{2} + \frac{1}{3} + \ldots + \frac{1}{n-1} \right)^k .
    \]
    In this expression, every product as above appears $k!$ times
    (and there are also additional products with non-distinct factors). 
    Therefore, the total number of permutations with $k+1$ cycles is at most 
    \[
        \frac{n!}{n \cdot k!} \cdot 
        \left( \frac{1}{1} + \frac{1}{2} + \frac{1}{3} + \ldots + \frac{1}{n-1} \right)^k \, ,
    \]
    implying the desired result
    by the standard bound on the harmonic series.  
\end{proof}

\begin{proof}[Proof of Theorem~\ref{t:all_lower}]
    For simplicity, replace the upper bound in Proposition~\ref{p31} by the slightly larger bound
    \[
        p_k := 
        \frac{(\ln n + 1)^k}{n \cdot k!} .
    \]
    Recalling that $k! > (k/e)^k \cdot \sqrt{2 \pi k}$ for every $k\ge 1$, 
    Proposition~\ref{p31} implies that 
    the probability of having $k_0 +1$ cycles, 
    for $k_0 := \lceil e (\ln n + 1) \rceil$, is at most
    \[
        p_{k_0} 
        < \frac{1}{n \cdot \sqrt{2 \pi k_0}} \cdot \left( \frac{e (\ln n + 1)}{k_0} \right)^{k_0}
        \le \frac{1}{n \cdot \sqrt{2 \pi k_0}} \, .
    \]
    Now note that, by definition, for $k \ge k_0$
    \[
        p_{k+1} 
        = p_k \cdot \frac{\ln n + 1}{k + 1}
        \leq p_k \cdot \frac{\ln n + 1}{e(\ln n + 1)}
        = p_k \cdot \frac{1}{e} \, .
    \]
    It follows that the probability of a random permutation in $S_n$ to have more than $k_0$ cycles is at most
    \[
        \sum_{k = k_0}^{n-1} p_k
        < p_{k_0} \cdot \sum_{m = 0}^{\infty} e^{-m}
        < \frac{1}{n} \cdot \frac{1}{\left( 1- e^{-1} \right) \cdot \sqrt{2 \pi k_0}}
        < \frac{1}{n} \, .
    \]
    Therefore, for a random permutation $\pi$, $\cyc(\pi c^j) \le k_0$ for every $j$ with high probability, and hence there exists such a permutation $\pi$.  
\end{proof}
\vspace{0.2cm}



\noindent
{\bf Acknowledgments.} 
We thank Jakob F\"uhrer for pointing out two relevant references,
and thank him and Oriol Sole Pi for helpful comments. 
Noga Alon is supported in part by NSF grant DMS-2154082.

\end{document}